\documentclass[11pt,letterpaper,twoside,english]{amsart}
\usepackage{hyperref}
\usepackage[foot]{amsaddr}
\usepackage{amsmath,amsthm,amssymb}
\usepackage{tikz,tikz-cd}
\usepackage{setspace}
\usepackage{comment}
\usepackage{graphicx}

\usepackage[backend=biber,style=alphabetic,maxnames=5]{biblatex}

\newtheorem{theorem}{Theorem}[section]
\newtheorem{prop}[theorem]{Proposition}
\newtheorem{lemma}[theorem]{Lemma}
\newtheorem{corollary}[theorem]{Corollary}
\theoremstyle{remark}
\newtheorem{definition}[theorem]{Definition}
\newtheorem{remark}[theorem]{Remark}

\numberwithin{equation}{section}

\numberwithin{equation}{section}

\def\bC{\mathbb{C}}
\def\bQ{\mathbb{Q}}
\def\bR{\mathbb{R}}
\def\bZ{\mathbb{Z}}
\def\bP{\mathbb{P}}
\def\cV{\mathcal{V}}
\def\cH{\mathcal{H}}
\def\cJ{\mathcal{J}}
\def\cF{\mathcal{F}}
\def\cC{\mathcal{C}}
\def\cM{\mathcal{M}}
\def\cE{\mathcal{E}}

\def\cD{\mathcal{D}}

\usepackage[mathscr]{eucal}

\usepackage{enumitem}

\begin{document}

\title{Remarks on the Griffiths infinitesimal invariant of algebraic curves}
\author{Haohua Deng}
\address{Department of mathematics at Duke University, 120 Science Drive, 117 Physics Building,
Campus Box 90320, Durham, North Carolina, 27708-0320}
\email{haohua.deng@duke.edu}

\date{\today}

\maketitle

\small
\begin{center}
\textbf{Abstract}
\end{center}

We study two normal functions defined on the moduli space of smooth genus $4$ algebraic curves including the Ceresa normal function. In particular, we study the vanishing criteria for the Griffiths infinitesimal invariants of both normal functions over a family of genus $4$ curves with certain trigonal structures.

\

\noindent\textbf{MSC classes: } 	14H10, 14H45, 14D07

\section{Introduction}\label{Sec01}

\subsection{Backgrounds}
Let $C$ be a non-hyperelliptic smooth algebraic curve of genus $g\geq 3$. In \cite{Cer83} Ceresa showed there is a canonically defined homologically trivial algebraic cycle in $\mathrm{CH}^{g-1}(\mathrm{Jac}(C))$ which is not algebraically trivial, thereby verifying the non-triviality of the corresponding Griffiths group. This cycle is known as the Ceresa cycle. In this article we will focus on the Griffiths Abel-Jacobi map \eqref{eqn:griffithsAJmap} associated to the Ceresa cycle and its induced Ceresa normal function over moduli spaces of curves, which we denote as $\nu_c$.


Given a family of curves over a quasi-projective base, we are particularly interested in the locus over which the Ceresa normal function $\nu_c$ has non-generic rank (Definition \ref{def:rankofnormalfunction}). This locus is shown to be algebraic by \cite{GZ24}. By studying the rank-dropping locus, we may be able to detect the positive-dimensional locus over which the Ceresa normal function is torsion (which was shown to be also algebraic by \cite{KT24} independently). One way to obtain a better picture of the rank-dropping locus is to study the Griffiths infinitesimal invariant $\delta\nu_c$ associated to $\nu_c$ (Section \ref{sec:grifinfinv}).

In this paper we are going to apply Collino--Pirola's adjunction formula \cite{CP95} to $\delta\nu_c$ for the genus $4$ case. Firstly, over all rank-$1$ deformations we study vanishing criteria for $\delta\nu_c$. Secondly, we study a specific family $\cC\rightarrow B$ of genus $4$ curves. We will show $\nu_c$ has maximal rank (See Section \ref{Sec03} for definitions) over this family, and moreover gives geometric control on the locus over which $\nu_c$ drops rank.

\subsection{Main results}

The first main result of this paper is to generalize part of Collino--Pirola's results to the genus $4$ case. Besides the Ceresa normal function $\nu_c$, there is another canonical normal function associated to family of genus $4$ curves defined by Griffiths in \cite[Sec. 6]{Gri83}, which we denote as $\nu_0$\footnote{Strictly speaking, $\nu_0$ is only well-defined over $\cM_4[2]$, the moduli space with an additional level-$2$ structure, but $\delta\nu_0$ is well-defined for any $C\in \cM_4$.}.
\begin{theorem}\label{thm:mainthm1}
    For a general non-hyperelliptic genus $4$ curve $C$, $\delta\nu_0$ and $\delta\nu_c$ have the same vanishing criteria for rank-$1$ deformations of $C$ as both of them vanish exactly along the Schiffer variations.
\end{theorem}
The main strategy for the proof is to find relations between the Griffiths infinitesimal invariants and the canonical embedding. In particular, we may understand $\delta\nu_c$ in terms of a certain projective hyperplane cutting off the canonical curve, generalizing a similar result of \cite[Thm 4.2.4]{CP95} for the genus $3$ case.

Moreover, by studying the relations between canonical embeddings and rank-$1$ first-order deformations of a curve, we are able to obtain the following corollary: 

\begin{corollary}\label{cor:maincorr1}
    For a general non-hyperelliptic genus $4$ curve $C$, $\delta\nu_{c,C}=0$ determines a degree-$6$ curve in $\bP(H^1(C,\mathcal{O}_C))$ which is exactly the canonical image of $C$.
\end{corollary}

While coming to general first-order deformations, the two normal functions $\nu_0$ and $\nu_c$ behave very differently. In the second part of the paper we study a certain family of genus $4$ curves which can be realized as a triple cover of the projective line branched over $6$ double points. We denote this $3$-dimensional family of curves as $\cC\rightarrow B\subset (\bP^1)^3$ (See Section \ref{Sec05}). Our second main theorem is about the behaviors of $\nu_0$ and $\nu_c$ over this family:
\begin{theorem}\label{thm:mainthm2}
 For the family of trigonal genus $4$ curves $\cC\rightarrow B$, the normal function $\nu_0$ as well as its Griffiths infinitesimal invariant $\delta\nu_0$ vanishes identically while the Ceresa normal function $\nu_c$ has maximal rank. Moreover, if $X\subset B$ is a subvariety on which $\nu_c$ is locally constant, then $\mathrm{dim}(X)\leq 1$ and the closure $\overline{X}$ of $X$ in $(\bP^1)^3$ must not intersect the smooth part of $(\bP^1)^3-B$.  
\end{theorem}
There is an equation defining a subbundle of $TB$ controlling $TX$ found by explicitly calculating the Kodaira--Spencer classes of the family, see Theorem \ref{thm:partialmainthm2}. To show the second part of the Theorem \ref{thm:mainthm2}, we need an alternative description on the rank of normal function which is the third main result of this paper: 
\begin{theorem}[Theorem \ref{Thm:foliationbyalgsubvar}]
    If a normal function $\nu$ associated with some variation of Hodge structures $\cV\rightarrow B$ has co-rank $r_c$, then there is a Zariski open subset $B^\circ\subset B$ foliated by $r_c$-dimensional algebraic subvarieties on each of which $\nu$ is locally constant. 
\end{theorem}
The proof of algebraicity of these $r_c$-dimensional submanifolds requires o-minimal geometry and its applications in Hodge theory. We give a brief survey in Section \ref{Sec07}. Next we will consider the algebraic monodromy group of the variation over each of these subvarieties. Theorem \ref{thm:mainthm2} will be deduced from a Zariski density argument.

\subsection{Some related works}

Normal functions over the moduli space of (stable) curves including the Ceresa normal function have been extensively studied; we refer readers to \cite{Ha13} for a comprehensive overview.

The rank of the Ceresa normal function $\nu_c$ over $\cM_g$ for $g\geq 3$ was proven to be maximal, independently by Gao--Zhang \cite{GZ24} and Hain \cite{Ha24}. As a generalization of \cite{CP95}, Pirola--Zucconi showed in \cite{PZ03} that over a subvariety of $\cM_g$ of dimension at least $2g-1$, the Griffiths infinitesimal invariant of $\nu_c$ must not vanish unless the subvariety is the hyperelliptic locus.

One may also be interested in finding specific families of curves on which $\nu_c$ drops rank, including those (positive-dimensional) families on which the Ceresa cycle itself is torsion. Some known examples include \cite{Lat23}, \cite{QZ24}. The relation between the Ceresa cycle class and its Abel--Jacobi image is also an interesting topic, see for example \cite{LS24}.


\medskip

\noindent\textbf{Acknowledgment.} The author sincerely appreciates Richard Hain for introducing the topic and sharing numerous motivating ideas. The author also thanks Matt Kerr and Colleen Robles for related discussions as well as the anonymous referee for their helpful feedback. Part of this work was motivated during the author's attendence at the ICERM conference "The Ceresa Cycle in Arithmetic and Geometry". The author thanks the organizers for their effort and generous funding support.

\section{Basic deformation theory}\label{Sec02}

Let $C$ be a smooth algebraic curve of genus $g$ which is not hyperelliptic. Recall that its first-order deformation space is $H^1(C, T_C)\simeq H^0(C,\Omega_C^{\otimes 2})^{\vee}$, which has dimension $3g-3$.

Consider the Noether map:
\begin{equation}\label{eqn:Noethermap}
  \rho: H^1(C, T_C)\rightarrow \ \mathrm{Hom}(H^0(C,\Omega_C), H^1(C,O_C))\simeq H^1(C, O_C)^{\otimes 2}
\end{equation}
which factors through $\mathrm{Sym}^2(H^{0,1}(C))$ and is injective by Noether's theorem. For any $\xi\in H^1(C, T_C)$, we say $\xi$ has rank $r$ if its image under Noether's map \eqref{eqn:Noethermap} has rank $r$. We also note that any rank-$1$ transformations has the form $\overline{\omega}\cdot \overline{\omega}\in \mathrm{Sym}^2(H^{0,1}(C))$ for some $\overline{\omega}\in H^{0,1}(C)$, and any rank-$2$ transformation has the form $\overline{\omega_1}\cdot \overline{\omega_2}$ for independent $\overline{\omega_1}, \overline{\omega_2}\in H^{0,1}(C)$.

Fix an orthonormal basis $(\omega_1,...,\omega_g)$ of $H^0(C, \Omega_C)\simeq H^{1,0}(C)$. The canonical embedding of $C$ is read as:
\begin{equation}\label{eqn:canonicalembedding}
    C\rightarrow \bP(H^{0,1}(C))\simeq \bP^{g-1}, \ p \rightarrow [\omega_1(p):...:\omega_g(p)].
\end{equation}
Consider the $2$-Veronese embedding $\bP(H^{0,1}(C))\rightarrow \bP(\mathrm{Sym}^2(H^{0,1}(C)))$ whose image may be identified with the projective space of rank-$1$ transformations in $\mathrm{Sym}^2(H^{0,1}(C))$. 
\begin{definition}\label{Def:SchifferVar}
We say a rank-$1$ transformation $\xi_p\in \bP(\mathrm{Sym}^2(H^{0,1}(C)))$ is a Schiffer variation at $p\in C$ if $\xi_p$ is the Veronese image of $p\in C$.
\end{definition}
\begin{prop}
    A rank-$1$ transformation $\xi$ is a Schiffer variation at $p\in C$ if and only if $p\in C$ is a base point of $W_{\xi}:=\mathrm{Ker}(\xi)\leq H^{1,0}(C)$ viewed as a linear subsystem of $|\Omega_C^1|$.
\end{prop}
\begin{proof}
    Any $\overline{\omega}\in \bP(H^{0,1}(C))$ gives a unique hyperplane $\xi_\omega\in \bP(H^{1,0}(C))$ representing the annihilator of $\overline{\omega}$. Therefore $\xi_\omega$ has a base point $p\in C$ if and only if $\overline{\omega}=\overline{\omega}_p$ comes from the canonical image of $p\in C$.
\end{proof}
In general, for an effective divisor $D$ on $C$, we say $\xi\in H^1(C, T_C)\simeq H^0(C,\Omega_C^{\otimes 2})^{\vee}$ is supported on $D$ if $\xi$ annihilates $H^0(C,\Omega_C^{\otimes 2}(-D))$. A Schiffer variation $\xi_p$ is thus supported on $D=[p]$.

If $g\geq 4$, $\bP(H^1(C,T_C))\subset \bP(\mathrm{Sym}^2(H^{0,1}(C)))$ is an intersection of hyperplanes. These hyperplanes cut the Veronese image of $\bP(H^{0,1}(C))$ along a set of quadric hypersurfaces in $\bP(H^{0,1}(C))$ containing the canonical image of $C$. We arrive at Griffiths' famous theorem:
\begin{theorem}[Griffiths]
    For a non-hyperelliptic curve $C$, every rank-$1$ transformation in $\bP(H^1(C,T_C))$ is a Schiffer variation if and only if the canonical image of $C$ is cut out by quadric hypersurfaces.
\end{theorem}

The dual of the map \eqref{eqn:Noethermap} factoring through $\mathrm{Sym}^2(H^{0,1}(C))$ gives 
\begin{equation}\label{eqn:dualnoethermap}
    \rho^\vee: \mathrm{Sym}^2H^{1,0}(C)\rightarrow H^1(C, \Omega_C^{\otimes2}).
\end{equation}
When $g=4$, the map $\rho^\vee$ has a one-dimensional kernel space whose generator defines a unique quadric surface $Q\subset \bP(H^{0,1}(C))\simeq \bP^3$ containing the canonical image of $C$.  In this case not every rank-$1$ transformation is a Schiffer variation.
\section{Normal functions and Griffiths infinitesimal invariants}\label{Sec03}

\subsection{Basic theory of normal functions}
Let $\mathcal{V}\rightarrow B$ be an integral polarized variation of Hodge structures ($\mathbb{Z}$-PVHS) of weight $-1$ with type $(V_\mathbb{Z}, Q, h^{p,q})$ defined over a quasi-projective base $B$. The intermediate Jacobian associated to $\mathcal{V}\rightarrow B$ is $\mathcal{J}(\mathcal{V}):=\frac{\mathcal{V}_\mathbb{C}}{\mathcal{F}^0\cV_{\bC}+\cV_\bZ}$. We also define
\begin{equation}
    \cJ_{h}(\cV):=\mathrm{ker}\{\overline{\nabla}: \cJ(\cV)\rightarrow \frac{\cV}{\cF^{-1}}\otimes \Omega_B^1\}
\end{equation}
as the horizontal part of $\cJ(\cV)$.
\begin{definition}
    A normal function $\nu$ assocated to $\cV\rightarrow B$ is a section of $\cJ_h(\cV)\rightarrow B$.
\end{definition}
In general, we only care about normal functions which are \textbf{admissible}. The precise definition can be found at for example \cite{Sai96}. Briefly speaking, admissibility means the existence of relative monodromy weight filtrations and nilpotent orbit theorem (\cite{Sch73}) along the boundary strata.

Let $\mathcal{X}\rightarrow B$ be a family of smooth projective varieties of dimension $n$ and $\mathcal{Z}\subset \mathcal{X}$ be a subvariety of codimension $r$ such that for each $b\in B$, $Z_b:=\mathcal{Z}\cap X_b\in \mathrm{CH}^{r}_{\mathrm{hom}}(X_b)$. For each $b\in B$ the Griffiths Abel--Jacobi map associates $Z_b$ an element in $J(H_b):=\frac{H_b(\bC)}{F^0H_b+H_b(\bZ)}$ where $H_b:=H^{2n-2r+1}(X_b)(n-r+1)$ defined as follows.

Using the identification:
\begin{equation}
    J(H_b):=\frac{H_b(\bC)}{F^0H_b+H_b(\bZ)}\cong \frac{(F^0H_b)^{\vee}}{H_b(\bZ)},
\end{equation}
there is the defining formula
\begin{equation}\label{Def:GriffAJmap}
    \nu(\omega(b))=\int_{\Gamma_b}\omega_b
\end{equation}
where $\omega(b)\in F^0H_b$ and $\Gamma_b\subset X_b$ is a $(2n-2r+1)$-chain satisfies $\partial \Gamma_b=Z_b$, both vary smoothly locally in $b$.

By patching together all intermediate Jacobians $J(H_b)$ and the Abel--Jacobi image given by $Z_b$, we obtain a normal function $\nu$ as a section of $\cJ\rightarrow B$. Such a normal function is said to be of geometric origin. The following theorem is a standard fact in Hodge theory, see for example \cite[Theorem 7.3]{BZ90}.
\begin{theorem}
    Normal functions of geometric origins are admissible.
\end{theorem}

\subsection{The rank of normal function}

In this subsection we fix a $\mathbb{Z}$-PVHS $\cV\rightarrow B$ of weight $-1$ and an admissible normal function $\nu\in H^0(B, \cJ(\cV))$. For any $b\in B$, we have the exact sequence:
\begin{equation}
    0\rightarrow V_b/F^0V_b\rightarrow T_{\nu(b)}\cJ(\cV)\rightarrow T_bB\rightarrow 0
\end{equation}
with the natural splitting $T_bB \rightarrow T_{\nu(b)}\cJ(\cV)$ given by the horizontal leaves of $\cV\rightarrow B$. Let $\pi_v$ be the associated projection $T_{\nu(b)}\cJ(\cV)\rightarrow V_b/F^0V_b$.
\begin{definition}\label{def:rankofnormalfunction}
    The rank of $\nu$ at a point $b\in B$ is defined to be the rank of the tangent map:
    \begin{equation}
       \pi_v\circ d\nu|_b: T_bB\rightarrow T_{\nu(b)}\cJ(\cV)\rightarrow V_b/F^0V_b.
    \end{equation}
    The rank of $\nu$ is defined to be the rank of $\nu$ at a very general point $b\in B$.
\end{definition}

\begin{remark}
    By using the identification
    \begin{equation}
        \cJ(\cV)\simeq \cJ(\cV_\bR)=\cV_\bR/\cV_\bZ,
    \end{equation}
    as $C^{\infty}$-manifolds, we may also define the rank of $\nu$ at $b\in B$ by half of the rank of the real tangent map:
    \begin{equation}
        \pi_v\circ d\nu_\bR|_b: (T_bB)_\bR\rightarrow T_{\nu(b)}\cJ(\cV_\bR)\rightarrow V_{b,\bR}.
    \end{equation}
\end{remark}
The following lemma is an immediate consequence from the definition:
\begin{lemma}
    $\nu$ is locally constant if and only if it has rank $0$. In this case for any $b\in B$, there exists a neighborhood $b\in U\subset B$ such that $\nu(U)\subset \cJ(\cV)$ admits a lift in $\Gamma^{\nabla}(U, \cV)$, the $\nabla$-flat sections of $\cV$ on $U$.
\end{lemma}

There will be more discussions about the normal function rank in Section \ref{Sec07}.
\subsection{The Griffiths infinitesimal invariant}\label{sec:grifinfinv}

Associated to a normal function $\nu$ is a differential invariant of $\nu$ studied by Griffiths \cite{Gri83}, denoted as $\tilde{\delta}\nu$ which is described as follows.

For $b\in B$, choose any local lift $v$ of $\nu$ around $b$, we have
\begin{equation}
\nabla v\in \mathrm{Ker}\{\Omega_B^1\otimes \cF^{-1}\xrightarrow{\nabla} \Omega^2_B\otimes \cF^{-2}\}. 
\end{equation}
Consider the Koszul complex:
\begin{equation}\label{eqn:fullkoszulcomplex}
    \cF^{0}\xrightarrow{\nabla} \Omega^1_b\otimes \cF^{-1}\xrightarrow{\nabla} \Omega^2_b\otimes \cF^{-2},
\end{equation}
since different choices of $v$ are differed by elements in $\cF^0$ and $\nabla$-flat sections, $\tilde{\delta}\nu_b$ gives a well-defined element in the Koszul cohomology group:
\begin{equation}\label{Eqn:Grifinfinv}
    \tilde{\delta}\nu_b\in \tilde{\mathbb{H}}_b:=\frac{\mathrm{Ker}\{\Omega^1_b\otimes \cF_b^{-1}\xrightarrow{\nabla} \Omega^2_b\otimes \cF_b^{-2}\}}{\mathrm{Img}\{\cF_b^{0}\xrightarrow{\nabla} \Omega^1_b\otimes \cF_b^{-1}\}}
\end{equation}
This is usually called the (full) Griffiths infinitesimal invariant of $\nu$ at $b\in B$. 

In practice a weaker invariant called the first Griffiths infinitesimal invariant $\delta\nu$ is more suitable for calculation. Consider the graded part of \eqref{eqn:fullkoszulcomplex}:
\begin{equation}\label{eqn:koszulcomplex}
    H_b^{0,-1}\xrightarrow{\overline{\nabla}} \Omega^1_b\otimes H_b^{-1,0}\xrightarrow{\overline{\nabla}} \Omega^2_b\otimes H_b^{-2,1},
\end{equation}
The first Griffiths infinitesimal invariant gives a well-defined element in:
\begin{equation}\label{Eqn:1stGrifinfinv}
    \delta\nu_b \in\mathbb{H}_b:=\frac{\mathrm{Ker}\{\Omega^1_b\otimes H_b^{-1,0}\xrightarrow{\overline{\nabla}} \Omega^2_b\otimes H_b^{-2,1}\}}{\mathrm{Img}\{H_b^{0,-1}\xrightarrow{\overline{\nabla}} \Omega^1_b\otimes H_b^{-1,0}\}}.
\end{equation}
By considering the dual complex of \eqref{eqn:koszulcomplex}:
\begin{equation}\label{eqn:dualkoszulcomplex}
    \wedge^2T_b\otimes H_b^{1,-2}\rightarrow T_b\otimes H_b^{0,-1}\xrightarrow{} H_b^{-1,0},
\end{equation}
$\delta\nu_b$ can also be seen as a linear function on:
\begin{equation}\label{Eqn:1stGrifinfinvdualrepresentation}
    \mathbb{H}_b^{\vee}:=\frac{\mathrm{Ker}\{T_b\otimes H_b^{0,-1}\xrightarrow{} H_b^{-1,0}\}}{\mathrm{Img}\{\wedge^2T_b\otimes H_b^{1,-2}\xrightarrow{} T_b\otimes H_b^{0,-1}\}}
\end{equation}
given by:
\begin{equation}
    \delta\nu_b(\sum\xi_i\otimes w_i)=\sum\langle \nabla_{\xi_i}v, w_i\rangle
\end{equation}
where $v$ is any local lift of $\nu$ around $b$.

The Griffiths infinitesimal invariant is a powerful tool to study the local behavior of a normal function. In particular, we have:
\begin{theorem}[\cite{Gri83}, Sec. 6(a)]
    The normal function $\nu$ is locally constant, or say has rank $0$, if and only if $\tilde{\delta}\nu_b= 0$ at every $b\in B$. 
\end{theorem}
\begin{remark}
It is clear that $\delta\nu=0$ is a necessary but in general not sufficient condition for $\tilde{\delta}\nu=0$. To fill the gap between them, in \cite{Gre89} Mark Green defined a sequence of infinitesimal invariants $\delta\nu=\delta_1\nu, \delta_2\nu,...,\delta_k\nu,...$ such that the vanishing of all of $\delta_k\nu$ holds if and only if $\tilde{\delta}\nu=0$. We will not use these invariants in this paper.
\end{remark}
For the rest of this paper, by the Griffiths infinitesimal invariant we shall mean the first Griffiths infinitesimal invariant.

\subsection{Infinitesimal invariant the Ceresa cycle normal function}

Fix a smooth genus $g$ algebraic curve $C$ and a marked point $p_0\in C$. The image of Abel--Jacobi map $C\rightarrow J(C), \ p \rightarrow \int_{p_0}^p$ in $J(C)$ gives an algebraic cycle $[C]-[C^-]\in Z^{g-1}(J(C))$. In \cite{Cer83} Ceresa showed for a general curve $C$,
\begin{equation}
    0\neq [C]-[C^-] \in \mathrm{Griff}^{g-1}(J(C)),
\end{equation}
where
\begin{equation}
    \mathrm{Griff}^{g-1}(J(C)):=\frac{\mathrm{Ch}^{g-1}_{\mathrm{hom}}(J(C))}{\mathrm{Ch}^{g-1}_{\mathrm{alg}}(J(C))}
\end{equation}
is the Griffiths group. The Griffiths Abel--Jacobi map \eqref{Def:GriffAJmap} gives:
\begin{equation}\label{eqn:griffithsAJmap}
    \mathrm{AJ}([C]-[C^-])\in J^{g-1}(J(C))\simeq \frac{F^2H^3(J(C))^{\vee}}{H_3(J(C), \bZ)}.
\end{equation}
The image depends on the base point $p_0\in C$, but its projection onto the primitive factor
\begin{equation}\label{eqn:primitivejacobian}
    PJ^{g-1}(J(C)):=\frac{F^2PH^3(J(C),C)^{\vee}}{H_3(J(C),\bZ)}
\end{equation}
does not (\cite{CP95}, Prop. 2.2.1). Therefore, associated to the universal curve of genus $g$: $\cM_{g,1}\rightarrow \cM_g$ is a $\bZ$-PVHS $\wedge^3_0\cV\rightarrow \cM_g$ with fibers being the primitive cohomology $PH^3(J(C),C)(2)$ and associated admissible normal function $\nu_c$ given by the Ceresa cycle. By the existence of the Kuranishi family, $\delta\nu_{c,C}$ is well-defined and called the Griffiths infinitesimal invariant of $C$.

More precisely, the family of intermediate Jacobians $\mathcal{PJ}\rightarrow \cM_g$ has fiber at $C$ given by \eqref{eqn:primitivejacobian}. Therefore $\delta\nu_{c,C}$ gives a linear function on $\mathbb{H}_b^{\vee}$ defined by equation \eqref{Eqn:1stGrifinfinvdualrepresentation}.

We introduce a result given by Collino-Pirola which is useful on computing the infinitesimal invariant. This gives a partial criteria on the degeneracy of $\nu_{c}$ at $C$. Let $\xi\in H^1(C,T_C)$ be a transformation with $\mathrm{dim}(W_\xi)\geq 2$, and $\sigma_1,\sigma_2\in H^{1,0}(C)$ are $2$ independent forms annihilated by $\xi$.
\begin{theorem}[\cite{CP95}, Sec. 2]\label{thm:collinopirolacoretheorem}
    $\delta\nu_{c,C}(\xi\otimes \sigma_1\wedge \sigma_2\wedge \overline{\omega})=0$ for any $\omega\in H^{1,0}(C)$ orthogonal to $\langle \sigma_1,\sigma_2\rangle$ if and only if $\xi$ is supported on the base locus of $\langle \sigma_1,\sigma_2\rangle$.
\end{theorem}
The proof relies on the construction of the adjunction map defined in \cite{CP95}, we will not cover it in details here.
\section{The Genus four case}\label{Sec04}

We will be focusing on the case $g=4$ in the rest of the paper. It is well-known that for a general smooth genus $4$ curve $C$, its canonical embedding into $\bP^{g-1}\simeq \bP^3$ is the intersection of a unique quadric $Q$ and a cubic $V$ well-defined up to cubics generated by $Q$.

To the family $\cM_4$ of smooth genus $4$ curves, let $\cM_{4,1}\rightarrow \cM_4$ be the universal curve and $\cV\rightarrow \cM_4$ be the associated $\bZ$-PVHS. There is a family of $0$-cycles defined up to a sign which leads to an admissible normal function $\nu_0$ underlying $\cV\rightarrow \cM_4[2]$ with the $2$-level structure dealing with ambiguity of the sign, and hence a well-defined infinitesimal invariant $\delta\nu_0$ for any $C\in \cM_4$. 
\subsection{The normal function $\nu_0$}

Since any quadric hypersurface $Q\subset \bP^3$ is isomorphic to $\bP^1\times \bP^1$, the two rulings $l_1, l_2$ of $Q$ cut the canonical curve $C=Q\cap V$ at the divisor
\begin{equation}\label{eqn:canonicalzerocycle}
    \pm D_0(C):=\sum_{1\leq i\leq 3}[p_i]-\sum_{1\leq i\leq 3}[q_i]
\end{equation}
which is homologous to zero and whose image in $\mathrm{CH}^1_{\mathrm{hom}}(C)$ is well-defined up to a sign. Therefore, $\nu_0$ is just the normal function associated to this family of cycles:
\begin{equation}
    \nu_0: \cM_4[2]\rightarrow \cJ(\cV)\simeq \frac{\cV}{\cF^0\cV+\cV_\bZ}.
\end{equation}

In \cite{Gri83} Griffiths studied the infinitesimal invariant $\delta\nu_0$ over $\cM_4$. Fix a non-hyperelliptic curve $C$, we have $T_CM_4\simeq H^1(C,T_C)$ via the Kodaira--Spencer map. Consider the set of rank-$1$ deformations of $C$:
\begin{center}
    $\cC_\Delta\rightarrow \Delta$, $\Delta:=\mathrm{Spec}\{\frac{\bC(\epsilon)}{\epsilon^2}\}$, $\cC_0\simeq C$,
\end{center}
such that the image of the Kodaira--Spencer class $\xi\in H^1(C,T_C)$ in $\mathrm{Hom}(H^{1,0}, H^{0,1})$ has rank $1$. Using the identification introduced in the beginning of Sec. 2, this set may be identified with the quadric hypersurface $Q\subset \bP^3$.
\begin{theorem}\cite[Sec. 6]{Gri83}\label{thm:deltanuzerovanishing}
    $\delta\nu_0|_\Delta$ vanishes if and only if $\xi\in C=Q\cap V$, equivalently if and only if $\xi$ is a Schiffer variation.
\end{theorem}
Indeed, he showed that $\delta\nu_0$ defines an element in $H^0(Q, \mathcal{O}_Q(3))$ whose vanishing locus is exactly the canonical curve $C$.
\subsection{The Ceresa normal function $\nu_c$}
In this subsection we consider the Ceresa normal function $\nu_c$ over $\cM_4$ and its infinitesimal invariant $\delta\nu_c$. In particular, we show $\delta\nu_c$ and $\delta\nu_0$ have the same vanishing locus on rank-$1$ deformations. This will imply Theorem \ref{thm:mainthm1} and Corollary \ref{cor:maincorr1}.

Let $\cM_4\rightarrow \mathcal{A}_4$ be the period map where $\mathcal{A}_4$ is the moduli space of principally polarized abelian varieties of dimension $4$. The image of $\cM_4$ is known as the Jacobian divisor of $\mathcal{A}_4$. Let $C\in \cM_4$ and $T_C\mathcal{A}_4\simeq \mathrm{Sym}^2(H^{0,1}(C))$. Denote $\{P^{p,q}, \ p+q=-1\}$ as the Hodge decomposition of $PH^3(J(C),C)(2)$. We have the commutative diagram:
\begin{equation}\label{eqn:doublekoszulcomplexgenus4}
\begin{tikzcd}
\wedge^2T_C\cM_4\otimes P^{1,-2} \arrow[d] \arrow[r] & T_C\cM_4\otimes P^{0,-1} \arrow[d] \arrow[r]& P^{-1,0} \arrow[d, "="] \\
\wedge^2T_C\mathcal{A}_4\otimes P^{1,-2} \arrow[r] & T_C\mathcal{A}_4\otimes P^{0,-1} \arrow[r] & P^{-1,0}.
\end{tikzcd}
\end{equation}
Following the computation in \cite[Sec. 7]{No93}, the bottom row of \eqref{eqn:doublekoszulcomplexgenus4} is exact, while the top row has the cohomology group $\mathbb{H}_C^{\vee}$ defined in \eqref{Eqn:1stGrifinfinvdualrepresentation} (We regard the curve $C$ as a point in $\cM_4$).

Let $\mathcal{K}:=\mathrm{Ker}\{T_C\mathcal{A}_4\otimes P^{1,-2}\rightarrow P^{0,-1}\}$. It follows that $\mathcal{K}\simeq \mathrm{Sym}^3(H^{0,1}(C))$. Choose an orthonormal basis $\{\omega_i, \ 1\leq i\leq 4\}$ of $H^{1,0}(C)$. The equivalence is given on decomposible tensors by
\begin{equation}\label{eqn:symspacetokernel}
    \overline{\omega_1}\cdot\overline{\omega_2}\cdot\overline{\omega_3}\in \mathrm{Sym}^3(H^{0,1}(C)) \rightarrow \sum_{\sigma\in S_3} \overline{\omega_{\sigma(1)}}\cdot\overline{\omega_{\sigma(2)}}\otimes \star\omega_{\sigma(3)}\in \mathcal{K}
\end{equation}
where $\star \omega$ is the Hodge-star operator on $H^{1,0}(C)$ with the chosen orthonormal basis. Let $\mathcal{K}^\circ\subset \mathcal{K}$ be the subspace $\mathcal{K}\cap (T_C\cM_4\otimes P^{1,-2})$, hence $\mathcal{K}^\circ$ may be regarded as a subspace of $\mathrm{Sym}^3(H^{0,1}(C))$.

Take any $0\neq\eta\in T_C\mathcal{A}_4-T_C\cM_4$ which gives a decomposition $T_C\mathcal{A}_4=T_C\cM_4\oplus \langle\eta\rangle$. There is a well-defined map $\rho_\eta: \mathcal{K}^\circ\rightarrow \mathbb{H}_C^{\vee}$ given by:
\begin{equation}\label{eqn:maptothekoszulcohmspace}
    \sum\xi_i\otimes \omega_i\in \mathcal{K}^\circ\xrightarrow{\rho_\eta}\sum\xi_i\otimes \nabla_\eta\omega_i\in \mathbb{H}_C^{\vee}
\end{equation}
with $\xi_i\in T_C\cM_4$. Since different choices of $\eta$ are differed by a constant multiple or an element in $T_C\cM_4$, the map $\rho_\eta$ with image in $\bP\mathbb{H}_C^{\vee}$ is canonically defined. By pulling back we may realize $\delta\nu_{c,C}$ as a linear functional on $\bP\mathcal{K}^\circ\subset \bP\mathrm{Sym}^3(H^{0,1}(C))$.

Note that for any rank-$1$ transformation $\xi=\overline{\omega}\cdot\overline{\omega}$, take $\{\sigma_1,\sigma_2,\sigma_3\}$ as an orthonormal basis of $W_\xi$, the maps \eqref{eqn:symspacetokernel} and \eqref{eqn:maptothekoszulcohmspace} are read as:
\begin{equation}\label{eqn:rank1deformationsymspace}
    \overline{\omega}\cdot\overline{\omega}\cdot\overline{\omega} \rightarrow \overline{\omega}\cdot\overline{\omega}\otimes \sigma_1\wedge\sigma_2\wedge\sigma_3\xrightarrow{\rho_\eta}\sum\overline{\omega}\cdot\overline{\omega}\otimes \nabla_\eta(\sigma_1\wedge\sigma_2\wedge\sigma_3).
\end{equation}
Since rank-$1$ deformations of $C$ are exactly given by points on the quadric surface $Q$ defined by the kernel of $\rho^\vee$ (see \eqref{eqn:dualnoethermap}), $\delta\nu_{c,C}$ gives an element in $H^0(Q,O_Q(3))$. The following proposition and Theorem \ref{thm:deltanuzerovanishing} imply the main Theorem \ref{thm:mainthm1}.
\begin{prop}
    For a general genus $4$ curve $C$, $\delta\nu_{c,C}$ vanishes on the image of a rank-$1$ deformation $\xi$ under \eqref{eqn:rank1deformationsymspace} if and only if $\xi=\xi_p$ is a Schiffer variation.
\end{prop}
\begin{proof}
    We first show $\delta\nu_{c,C}$ does not vanish identically on the Veronese image of $\bP(H^{0,1}(C))$. It suffices to find one curve $C$ and one rank-$1$ deformation $\xi\in  H^1(C,T_C)$ such that $\delta\nu_{c,C}$ does not vanish on the image of $\xi$ under \eqref{eqn:rank1deformationsymspace}.

    We consider the same example as \cite[Sec. 6(d)]{Gri83}. Take $C$ and $p\in C$ satisfy that the two rulings $l_1,l_2$ of $Q$ at $p$ have triple and double tangent point with $C$ at $p$ respectively. The plane generated by $l_1,l_2$ in $\bP(H^{0,1}(C))$ represents a holomorphic form $\omega_0\in H^0(C,\Omega_C(3p))$. Take another two forms $\omega_1,\omega_2$ such that $\omega_0\cap \omega_i=l_i$, and $\omega_0,\omega_1,\omega_2$ are mutually orthogonal. Take $q\in \omega_1\cap\omega_2-\omega_0$. The base locus of $\langle \omega_1,\omega_2\rangle$ is $3[p]$ while $\xi_q$ is not supported on $3[p]$ as it does not annihilate $\omega_0$. Theorem \ref{thm:collinopirolacoretheorem} implies $\delta\nu_{c,C}(\omega_1\wedge\omega_2\wedge\overline{\zeta})$ does not vanish for all $\zeta\in H^{1,0}(C)$.

    It remains to show for any $p\in C$ and the corresponding Schiffer variation $\xi_p\in \bP H^1(C,T_C)$, $\delta\nu_{c,C}$ vanishes on the image of $\xi_p$ under \eqref{eqn:rank1deformationsymspace}, but this follows from the fact $\delta\nu_{c,C}(\xi_p\otimes\sigma_i\wedge\sigma_j\wedge\nabla_\eta\sigma_k)$ vanishes as $p$ is a base point of $\langle\sigma_1,\sigma_2,\sigma_3\rangle$.
\end{proof}
\section{A family of special trigonal curves}\label{Sec05}

Denote $\bP$ as the weighted projective plane $\bP[1:1:2]$. Consider the family of smooth projective curves
\begin{equation}\label{eqn:aspecialfamilygenus4curves}
    \cC:=\{(b_1,b_2,b_3),[X:Z:Y])\in B \times \bP \ | \ Y^3=(X^3-Z^3)\prod_{1\leq j\leq 3}(X-b_jZ)\},
\end{equation}
where
\begin{equation}\label{eqn:boundarydivisors}
    B:= (\bP^1)^3 - \cup_{i,j}\{b_i=b_j\} - \cup_{i}\{b_i^3=1\}.
\end{equation}
Hence $B$ is isomorphic to $\cM_{0,6}$, the moduli space of $6$ distinct marked points on $\bP^1$ up to projective equivalence. For $b=(b_1,b_2,b_3)\in B$, the curve $C_b$ has affine equation
\begin{equation}
    C_b: \overline{\{(x,y)\in \bC^2 \ | \ y^3=(x^3-1)\prod_{1\leq j\leq 3}(x-b_j) \}}\subset \bP.
\end{equation}
It is clear $C_b$ admits a $3$ to $1$ covering map to $\bP^1$ branched along $6$ double points. By Riemann-Hurwitz, it has genus $4$. Moreover it admits a degree $3$-automorphism $\rho: y\xrightarrow{} e^{\frac{2\pi i}{3}}y$. 

\begin{remark}
    By \cite{Loo24}, the family $\cC\rightarrow B$ parametrizes curves admitting a $g^1_3$ whose discriminant divisor on $\bP^1$ is the sum of $6$ (distinct) double points.
\end{remark}

To show the main Theorem \ref{thm:mainthm2}, we independently check the behaviors of $\delta\nu_0$ and $\delta\nu_c$ over the family $\cC\rightarrow B$.
\subsection{Rational triviality of the cycle $D_0$}
For $b\in B$ there is a canonical basis of $H^0(C_b, \Omega^1_{C_b})$ compatible with the eigenspace decomposition of the induced action from $\rho$:
\begin{equation}
    H^0(C_b, \Omega^1_{C_b})=\langle\frac{dx}{y}\rangle_{e^{\frac{4\pi i}{3}}} \oplus \langle\frac{dx}{y^2}, \frac{xdx}{y^2}, \frac{x^2dx}{y^2}\rangle_{e^{\frac{2\pi i}{3}}}=:\langle \omega_0\rangle\oplus \langle \omega_1, \omega_2, \omega_3\rangle.
\end{equation}
Take this ordered basis and consider the corresponding canonical embedding $C_b\xrightarrow{\phi} \bP(H^{0,1}(C_b))\simeq \bP^3$, the canonical image $\phi(C_b)$ lies on the quadric surface
\begin{equation}
 Q:=\{[z_0:z_1:z_2:z_3]\in \bP^3 \ | \ z_2^2-z_1z_3=0\}.
\end{equation}
The two rulings are given by the lines:
\begin{eqnarray}
    &L(t_1):=\{z_2+z_1=t_1(z_3-z_1)\}\cap \{z_1=t_1(z_2-z_1)\}\\
    &L(t_2):=\{z_2+z_1=t_2z_1\}\cap \{z_3-z_1=t_2(z_2-z_1)\}
\end{eqnarray}
for $t_i\in \bP^1$. Take $t_1, t_2$ such that  
\begin{equation}
    \frac{1}{t_1}+1=t_2-1,
\end{equation}
it follows that $[L(t_1)\cap C_b]=[L(t_2)\cap C_b]$. Hence for any $b\in B$, $t_1, t_2\in \bP^1$, the cycle $D_0(C_b)=[L(t_1)\cap C_b]-[L(t_2)\cap C_b]$ on $C_b$ is rationally trivial\footnote{This also implies $\nu_0$ can be defined over $\cC\rightarrow B$ without adding any level structures.}.

\subsection{The Ceresa cycle}
To study the Ceresa normal function of the family, we need to compute the explicit Kodaira--Spencer image of $T_bB$. Clearly
\begin{equation}
    T_bB=\{a_1\partial_1+a_2\partial_2+a_3\partial_3, \ \partial_j:=\frac{\partial}{\partial b_j}, \ a_j\in \bC\}\simeq \bC^3.
\end{equation}
To find the Kodaira--Spencer image of $\partial_j$, we notice that 
\begin{eqnarray}
    &\partial_j\omega_0=\frac{1}{3(x-b_j)}\omega_0,\\
    &\partial_j\omega_k=\frac{2}{3(x-b_j)}\omega_k, \ 1\leq k\leq 3.
\end{eqnarray}
It follows that
\begin{eqnarray}
    &\partial_j(\omega_2-b_j\omega_1)=\frac{2}{3}\omega_1,\\
    &\partial_j(\omega_3-b_j\omega_2)=\frac{2}{3}\omega_2
\end{eqnarray}
are both holomorphic, therefore $\langle\omega_2-b_j\omega_1, \omega_3-b_j\omega_2 \rangle$ is annihilated by the Kodaira--Spencer class of $\partial_j$. 

Moreover, it is clear for any $\xi\in T_bB$,
\begin{equation}
    \nabla_\xi H^0(C_b, \Omega^1_{C_b})_{e^{\frac{2\pi i}{3}}}\subset H^1(C_b, \mathcal{O}_{C_b})_{e^{\frac{2\pi i}{3}}}=\langle\overline{\omega_0}\rangle
\end{equation}
and similarly
\begin{equation}
    \nabla_\xi H^0(C_b, \Omega^1_{C_b})_{e^{-\frac{2\pi i}{3}}}=\nabla_\xi\langle\omega_0\rangle\subset H^1(C_b, \mathcal{O}_{C_b})_{e^{-\frac{2\pi i}{3}}}.
\end{equation}
As a consequence we have the following lemma.
\begin{lemma}
    For any $0\neq \xi\in T_bB$, its Kodaira--Spencer class in $H^1(C_b,T_{C_b})$ has rank $2$, and the two dimensional kernel $W_\xi\leq H^0(C_b, \Omega^1_{C_b})_{e^{\frac{2\pi i}{3}}}$.
\end{lemma}
We compute the explicit Kodaira--Spencer classes of $\partial_j$. For convenience denote 
\begin{equation}
    Q(x)=Q_b(x):=(x^3-1)(x-b_1)(x-b_2)(x-b_3).
\end{equation}
Note that though $\partial_j\omega_k$ is a $(1,0)$-form defined over $C_b-\{y=0\}$, it has zero residue around every branched point $p\in\{y=0\}$. By the Gysin sequence
\begin{equation}
    0\rightarrow H^1(C_b)\rightarrow H^1(C_b-\{p\})\xrightarrow{\mathrm{Res}}H^2(C_b,C_b-\{p\})\simeq H^0(\{p\})=\bC,
\end{equation}
there is a $1$-form defined over $C_b$ representing the cohomology class of $\partial_j\omega_k$ which we denote as $[\partial_j\omega_k]$. We compute the component of $[\partial_j\omega_k]$ in $H^{0,1}(C_b)$. 
\begin{lemma}\label{Lemma:KSmapclasscompute}
The following equations hold:
\begin{eqnarray}\label{eqn:intpairing}
    &\int_{C_b}\omega_0\wedge [\partial_j\omega_1]=\frac{6\pi i}{Q'(b_j)},\\
    &\int_{C_b}\omega_1\wedge [\partial_j\omega_0]=\frac{6\pi i}{Q'(b_j)},\\
    &\int_{C_b}\omega_0\wedge [\partial_j\omega_0]=0,\\
    &\int_{C_b}\omega_1\wedge [\partial_j\omega_k]= 0, \ 1\leq k\leq 3.
\end{eqnarray}
\end{lemma}
\begin{proof}
The calculation is similar to \cite[Chap. 1.1]{CSP17}, we first solve $\int_{C_b}\omega_0\wedge [\partial_1\omega_1]$. Note that around $b_1$, 
\begin{equation}
    \partial_1\omega_1=\frac{dy}{(x-b_1)Q^{'}(x)}\sim \frac{dy}{y^3} + *
\end{equation}
where $*$ represents some higher-order terms. 
Let $b_1\in U\subset C_b$ be a small enough open neighborhood of $b_1\in C_b$ on which $y$ is a local coordinate. Since both $\omega_0$ and $\partial_1\omega_1$ are holomorphic over $C_b-U$,
\begin{eqnarray}
    &\int_{C_b}\omega_0\wedge [\partial_1\omega_1]=\int_{U}\omega_0\wedge [\partial_1\omega_1]\\
    &=-\int_{U}\frac{3ydy}{Q^{'}(x)}\wedge [\frac{dy}{y^3}]\\
    &=\int_{U}\frac{3dy}{Q^{'}(x)}\wedge [d(\frac{1}{y})]\\
    &=\int_{U}d[\frac{3dy}{Q^{'}(x)y}]\\
    &=\int_{\partial U}\frac{3dy}{Q^{'}(x)y}\\
    &=2\pi i\mathrm{Res}_{y=0}\frac{3}{Q^{'}(x)y}=\frac{6\pi i}{Q^{'}(b_1)}.
\end{eqnarray}
The proof of the remaining equations are parallel.
\end{proof}
A direct consequence of Lemma \ref{Lemma:KSmapclasscompute} is:
\begin{prop}\label{prop:rankofKSmap}
    For any $0\neq \xi\in T_bB$, its Kodaira--Spencer image has rank $2$. Moreover, it sends $\langle\omega_k, \ 1\leq k\leq 3\rangle$ to $\langle\overline{\omega_0}\rangle$ and $\omega_0$ to $\langle\overline{\omega_\xi}\rangle$, where $\omega_\xi$ is a holomorphic form in $\langle\omega_k, \ 1\leq k\leq 3\rangle$ orthogonal to $W_\xi:=\mathrm{Ker}(\xi)$.
\end{prop}
For any $\xi\in T_bB$, suppose $W_\xi=\langle\sigma_1,\sigma_2\rangle\leq H^{1,0}(C_b)$. We consider whether $\delta\nu_{c,C_b}$ vanishes on the first deformation of $C_b$ given by $\xi$. If we can show that $\xi$ is not supported on the base locus of $W_\xi$, then Theorem \ref{thm:collinopirolacoretheorem} implies $\delta\nu_{c,C_b}$ does not vanish on the first-order deformation of $C_b$ given by $\xi$.

By \eqref{eqn:intpairing}, for $\xi=\sum_{1\leq j\leq 3} a_j\partial_j$, 
\begin{equation}
    W_\xi=\{\sum_{1\leq l\leq 3}\alpha_l\omega_l \ |  \sum_{1\leq j,l\leq 3}\frac{\alpha_la_jb_j^{l-1}}{Q'(b_j)}=0 \}.
\end{equation}
On the other hand, for a form $\omega:=\frac{\alpha_3x^2+\alpha_2x+\alpha_1}{y^2}dx, \ \alpha_3\neq 0$, by \cite[Sec. 3]{McM13}, the zero locus is precisely given by the roots of $\alpha_3x^2+\alpha_2x+\alpha_1=0$. 

To sum up, for $C=C_b, \ b=(b_1,b_2,b_3)$, $\xi=a_1\partial_1+a_2\partial_2+a_3\partial_3\in T_bB$ makes $W_\xi$ have non-empty base locus if and only if 
\begin{equation}\label{eqn:diffeqnpositivefiber}
    [a_1,a_2,a_3]A\in \{[X:Y:Z]\in \bP^2 \ | \ XZ-Y^2=0\},
\end{equation}
where
\begin{equation}
    A=
\begin{pmatrix}
\frac{1}{Q'(b_1)} & \frac{b_1}{Q'(b_1)} & \frac{b_1^2}{Q'(b_1)} \\
\frac{1}{Q'(b_2)} & \frac{b_2}{Q'(b_2)} & \frac{b_2^2}{Q'(b_2)} \\
\frac{1}{Q'(b_3)} & \frac{b_3}{Q'(b_3)} & \frac{b_3^2}{Q'(b_3)} 
\end{pmatrix}.
\end{equation}

Since this is an irreducible plane quadric which does not contain any linear subspaces of $\bP^2$ other than points, we may conclude:
\begin{theorem}\label{thm:partialmainthm2}
The Ceresa normal function $\nu_c$ has rank $\geq 2$ for the family $\cC\rightarrow B$. Moreover, if $L\subset B$ is a curve on which $\nu_c$ is locally constant, then the tangent bundle $TL$ must contained in the subbundle of $TB$ determined by \eqref{eqn:diffeqnpositivefiber}.
\end{theorem}

By Theorem \ref{Thm:GaoZhangthm}, There exist two algebraic subvarieties $B_3, B_2$ of $B$ such that for any $b\in B_i$, the rank of $\nu_c$ at $b$ is $i$. In particular, $B_3$ is either empty or Zariski dense in $B$, and $B_2$ is Zariski closed in $B$. Moreover, by Theorem \ref{Thm:foliationbyalgsubvar}, $B_2$ admits a foliation by algebraic curves over each of which $\nu_c$ is locally constant.

To finish proving Theorem \ref{thm:mainthm2} we need two important ingradients: Definability of normal functions and the algebraic monodromy group, both are introduced in details in the Appendix \ref{Sec07}.

\subsection{The rank of $\nu_c$}

To show the rank of $\nu_c$ is maximal (or equivalently, $B_3\neq \emptyset$), we first look at the algebraic monodromy group (see Definition \ref{def:algmonogroup}) of the $\bZ$-PVHS $\cV\rightarrow B$ given by the family of curves \eqref{eqn:aspecialfamilygenus4curves}. Let $\Gamma_0$ be its monodromy group. \cite[Thm 2.1]{Xu18} implies:
\begin{prop}\label{prop:bigmonodromy}
    The $\bR$-algebraic monodromy group $\overline{\Gamma_0}^{\bR}$ of $\cV\rightarrow B$ is isomorphic to $\mathrm{SU}(3,1)$. In particular, $\wedge^3_0V$ does not contain any $1$-dimensional $\overline{\Gamma}^{\bR}$-sub-representation.
\end{prop}
\begin{proof}
    Using the fundamental weights of $\mathfrak{sl}_4(\bC)$, $V\simeq V^{\omega_1+\omega_3}$, we have a decomposition of $\wedge^3V$ as its $\mathrm{SU}(3,1)$-irreducible representations:
    \begin{equation}
       \wedge^3V\simeq V^{\omega_1+\omega_2}\oplus V^{\omega_3+\omega_2}\oplus V^{\omega_1}\oplus V^{\omega_3}  
    \end{equation}
    Clearly there is no $1$-dimensional summand.
\end{proof}

Denote $\Phi_c: B\rightarrow \Gamma\backslash \mathcal{D}$ as the mixed period map associated to the Ceresa normal function $\nu_c$, and $\Phi_c^0: B\rightarrow \Gamma_0\backslash D$ be the pure period map given by projecting $\Phi_c$ to its graded pieces. See \eqref{eqn:normalfunctionmixedperiodmap}.

We claim that to show $\mathrm{rank}(\nu_c)=3$, it is enough to show $\mathrm{dim}(\Phi_c^0(B))=3$. This is because Proposition \ref{prop:bigmonodromy} implies the $\bZ$-PVHS $\wedge^3_0\cV\rightarrow B$ satisfies the assumption of Theorem \ref{Thm:GaoZhangthmrankpart}. To show $\mathrm{dim}(\Phi_c^0(B))=3$, it is enough to show the associated infinitesimal Torelli map is injective, that is:
\begin{prop}
For any $b\in B$, the map:
\begin{equation}\label{}
        d\Phi_c^0:T_bB\xrightarrow{}H^1(C_b, T_{C_b})\rightarrow \oplus_{k\in \bZ} \mathrm{Hom}(\cF^k(\wedge^3_0V_b), \cF^{k-1}(\wedge^3_0V_b)/\cF^k(\wedge^3_0V_b))
\end{equation}
is injective.
\end{prop}
\begin{proof}
    It is enough to show the component of $d\Phi_c^0$ with $k=1$ is injective. Note that $\cF^1(\wedge^3_0V_b)$ is genereted by:
    \[\{\omega_i\wedge\omega_j\wedge\omega_k, \ 0\leq i,j,k\leq 3\}.\]
    Take any $0\neq \xi\in T_bB$, Proposition \ref{prop:rankofKSmap} implies we may denote $W_\xi=\{\eta_1, \eta_2\}$ with $W_\xi\perp\langle \omega_0\rangle$. Hence $\xi(\omega_0\wedge \eta_1\wedge\eta_2)\in\eta_1\wedge\eta_2\wedge H^{0,1}(C_b)$ which is non-zero in $\cF^{0}(\wedge^3_0V_b)/\cF^1(\wedge^3_0V_b)$ as well as primitive.
\end{proof}
This concludes the proof for rank$(\nu_c)=3$.
\subsection{The subvariety $B_2\subset B$}

We now prove the statement in Theorem \ref{thm:mainthm2} about the closure of a subvariety of $B$ over which $\nu_c$ is locally constant. Such a variety must be an algebraic curve in $B$ as well as a leaf of the foliation over $B_2$ (Theorem \ref{Thm:foliationbyalgsubvar}). Take a general leaf $\zeta\subset B_2$. Denote $T:=\bP^3-B$, and $\zeta', T'$ as the image of $\zeta, T$ in $\bP^3\simeq \mathrm{Sym}^3\bP^1$. Since $\overline{\zeta'}\cap T'\neq \emptyset$, $\overline{\zeta}\cap T\neq \emptyset$. Suppose $\overline{\zeta}$ intersects $T$ at a smooth point of $T$. 
\begin{lemma}
    $\cV|_\zeta\rightarrow \zeta$ and $\cV\rightarrow B$ have the same $\bR$-algebraic monodromy group $\mathrm{SU}(3,1)$.
\end{lemma}
\begin{proof}
Note that $T$ is a union of $12$ hyperplanes which are all symmetric under the moduli sense, and the closure of every leaf contained in $B_2$ must intersect $T$ (as their images in $\bP^3$ do). It follows that the closure $\overline{\zeta}$ of a general leaf $\zeta\subset B_2$ must intersect each hyperplane in $T$ at a smooth point. This implies the monodromy group $\Gamma_{\zeta}$ over $\zeta$ has the same set of generators with the full monodromy group $\Gamma_0$ over $(\bP^1)^3$ given by the monodromy operators around the smooth part of each hyperplane in $T$. 
\end{proof}
On the other hand, since $\nu_c$ is locally constant on $\zeta$, there exists $w\in H^0(\zeta, V_\bR)$ lifting $\nu_c$, which means $w$ should be fixed by the monodromy group $\Gamma_\zeta$ which implies $\overline{\Gamma}^\bR_\zeta=\overline{\Gamma_0}^\bR=\mathrm{SU}(3,1)$ is contained in the stabilizer of $w$, a contradiction to Proposition \ref{prop:bigmonodromy}. This concludes the proof of Theorem \ref{thm:mainthm2}.

\subsection{Some remarks}

We are able to conclude a bit more: The proof showed that if there is an ($1$-dimensional) leaf $\zeta_0$ over which $\nu_c$ is locally constant, then the $\bR$-algebraic monodromy group $\overline{\Gamma}^\bR_{\zeta_0}$ must not be generic. Therefore, the positive-dimensional locus on which the Ceresa normal function is torsion must have special $\bQ$-algebraic monodromy group. 
\begin{corollary}
    The positive-dimensional locus in $B$ on which the Ceresa normal function is torsion must contain in the (algebraic) weakly special subspace of $B$ (in the sense of \cite[Sec. 6.3]{BBKT24}). 
\end{corollary}
Indeed, by \cite[Sec. 4]{QZ24}, the Ceresa cycle is torsion over the $1$-dimensional sub-family of \eqref{eqn:aspecialfamilygenus4curves}:
\begin{equation}
    \mathcal{L}:=\{(a,[X:Z:Y])\in (\bP^1-\{0,1,\infty\}) \times \bP \ | \ Y^3=(X^3-Z^3)(X^3-aZ^3)\},
\end{equation}
hence $\nu_c$ must be locally constant along this family. The base of this family is a union of lines $L\subset B_2$ whose closure only intersect $T$ at the $0$-dimensional strata. It is not clear whether $L$ is a union of irreducible components of $B_2$ or not.
\section{Appendix: Definability of normal functions}\label{Sec07}

In this section we survey the theory of normal functions using the mixed Hodge theory aspects. We also introduce the application of o-minimal geometry in Hodge theory. As a consequence, we prove a theorem regarding the locus on which a normal function is locally constant.

Some references for this sections are \cite{BKT20}, \cite{BBT22} and \cite{BBKT24}.

\subsection{Normal function as variation of mixed Hodge structures}

Regarding basic facts of admissible integral polarized variation of mixed Hodge structures ($\bZ$-PVMHS), we refer readers to \cite[Sec. 3-4]{BBKT24}. In this section we assume all normal functions and VMHS are admissible.

A normal function $\nu$ underlying a $\bZ$-PVHS $\cV\rightarrow B$ can be regarded as a $\bZ$-PVMHS $(\cE, \mathcal{W}, \cF)$ with only $2$ non-trivial graded quotients, with:
\begin{eqnarray}\label{eqn:VMHSfornormalfunction}
    &\mathrm{Gr^{\mathcal{W}}_{-1}\cE}, \cF(\mathrm{Gr^{\mathcal{W}}_{-1}\cE})\simeq (\cV, \cF)\\
    &\mathrm{Gr^{\mathcal{W}}_{0}\cE}\simeq \bZ(0),
\end{eqnarray}
where $\bZ(0)$ is the Tate Hodge structure. In other words, the exact sequence
\begin{equation}\label{eqn:exactsequenceZMHS}
    0\rightarrow \cV\rightarrow \cE\rightarrow \bZ(0)\rightarrow 0
\end{equation}
realizes $\nu$ as an element in $\mathrm{Ext}^1_{\bZ-\mathrm{VPMHS}}(\bZ(0), \cV)$. Note that the sequence \eqref{eqn:exactsequenceZMHS} always splits over $\bR$. We say the normal function is vanishing (resp. torsion) if the corresponding sequence \eqref{eqn:exactsequenceZMHS} splits over $\bZ$ (resp. $\bQ$).

Let $E=\cE_b$ for some reference point $b\in B$. Let $\bold{G}$ be the $\bQ$-algebraic subgroup of $\mathrm{Aut}(E)$ preserving the weight filtration $W=\mathcal{W}_s$ as well as corresponding graded polarizations, and $\bold{U}$ be its unipotent radical. Let $\cD$ be the mixed period domain parametrizing all $\bZ$-PVMHS on $E$ with the type given by \eqref{eqn:VMHSfornormalfunction}.

Let $\cD_\bR\subset \cD$ be the subdomain of all $\bR$-split members in $\cD$, $D$ be the period domain for the underlying $\bZ$-PVHS $\cV$, and $\mathcal{S}(W)$ be the
$\bR$-splitting variety of $W$. 
\begin{prop}
The following properties hold.
    \begin{enumerate}
        \item $\bold{G}(\bR)\bold{U}(\bC)$ acts transitively on $\cD$.
        \item $\cD=\cD_\bR$.
        \item $\cD_\bR\simeq D\times \mathcal{S}(W)$.
    \end{enumerate}
\end{prop}
\begin{proof}
    (2) comes from the observation \eqref{eqn:exactsequenceZMHS} always splits over $\bR$. See \cite[Sec. 3-4]{BBKT24} for (1) and (3).
\end{proof}
\begin{remark}
    (2) is false in general, for example when $\cD$ contains a mixed Hodge structure not splitting over $\bR$.
\end{remark}
The normal function $\nu$ thus give rise to a mixed period map and its associated pure period map by taking the weight quotients:
\begin{equation}\label{eqn:normalfunctionmixedperiodmap}
    \begin{tikzcd}
    B \arrow[r, "\Phi"] \arrow[rr, bend left=30, "\Phi_0"] & \Gamma\backslash \mathcal{D} \arrow[r] & \Gamma_0\backslash D
    \end{tikzcd}
\end{equation}
where $\Gamma_0$ is the monodromy group of $\cV\rightarrow B$ and $\Gamma\simeq \Gamma_0\ltimes V_\bZ$ is the monodromy group of $\cE\rightarrow B$ where $V$ is the fiber at a chosen reference point.

\begin{definition}\label{def:algmonogroup}
    Let $\mathbb{K}\leq \bC$ be a field, the $\mathbb{K}$-algebraic monodromy group $\overline{\Gamma_0}^{\mathbb{K}}$ of $\cV\rightarrow B$ is the identity connected component of the $\mathbb{K}$-Zariski closure of $\Gamma_0$.
\end{definition}

\begin{remark}
    Let $M\leq \mathrm{Aut}(V_\bQ)$ be the generic Mumford-Tate group of the $\bZ$-PVHS $\cV\rightarrow B$ defined over $\bQ$, then $\overline{\Gamma_0}^{\bQ}$ is a normal subgroup of $M^{\mathrm{der}}$, the derived subgroup of $M$, see \cite{And92}.
\end{remark}

\subsection{More on rank of normal functions}

We keep definitions and notations from Section \ref{Sec03}. Suppose $\mathrm{rank}(\nu)=r$ for some $0\leq r\leq \mathrm{dim}(B)$, denote $r_c:=\mathrm{dim}(B)-r$ as the co-rank of $\nu$.

\begin{theorem}\cite[Thm 5.2]{GZ24}\label{Thm:GaoZhangthmrankpart}
    Suppose the algebraic monodromy group of $\cV\rightarrow B$ is simple, and the variation of Hodge structure $\cV\rightarrow B$ has no locally constant summand, then 
    \begin{equation}
        r=\mathrm{min}\{\mathrm{dim}(\Phi(B)), \frac{1}{2}\mathrm{dim}(V)\}.
    \end{equation}
\end{theorem}

\begin{theorem}\cite[Thm 3.2]{GZ24}\label{Thm:GaoZhangthm}
    $B=B_r\supset B_{r-1}\supset...\supset B_0$, where 
    \begin{equation}
        B_k:=\{b\in B, \ \mathrm{rank}(\nu_b)\leq k\}, \ 0\leq k\leq r
    \end{equation}
    are algebraic subvarieties of $B$.
\end{theorem}

In other words, there exists a Zariski open subset $B^{\circ}\subset B$ such that $\mathrm{rank}(\nu_b)=r$ for every $b\in B^\circ$. Therefore after possibly replace $B$ by one of its Zariski open subset, we may assume $B_{r-1}=\emptyset$. The main result of this section is the following.
\begin{theorem}\label{Thm:foliationbyalgsubvar}
    $B$ admits a foliation by dimension-$r_c$ algebraic subvarieties on each of which the normal function $\nu$ is locally constant.
\end{theorem}
We first prove a weaker version of Theorem \ref{Thm:foliationbyalgsubvar}.
\begin{prop}\label{prop:foliationbysubmanifold}
    $B$ admits a foliation by dimension-$r_c$ complex submanifolds on each of which the normal function $\nu$ is locally constant, which we call the leaves of the foliation.
\end{prop}
\begin{proof}
 The co-rank of $\nu$ at $b\in B$ equals to $r_c$ implies in a local neighborhood $b\in U\subset B$, a local holomorphic lift $v\in H^0(U,\cV)$ and some flat local section $e\in H^0_{\nabla}(U, \cV)$ such that
 \begin{equation}
     \mathrm{dim}\ T_p(v\cap (e+\cF^0\cV|_U))=r_c
 \end{equation}
 for a generic point $p\in v\cap (e+\cF^0\cV)$. The projection of this intersection on $U$ gives a $r_c$-dimensional complex submanifold of $U$.

Therefore, the condition that the co-rank of $\nu$ at every point $b\in B$ is $r_c$ means there is a rank-$r_c$ integrable holomorphic distribution on $B$ whose integral submanifolds are $r_c$-dimensional submanifolds of $B$.
\end{proof}
Another ingredient needed to prove Theorem \ref{Thm:foliationbyalgsubvar} is o-minimal geometry, see \cite{BBT22} for an introduction. We will show the following proposition:
\begin{prop}\label{prop:leavesaredefinable}
    The $r_c$-dimensional leaves in Proposition \ref{prop:foliationbysubmanifold} are all definable in the o-minimal structure $\bR_{\mathrm{an, exp}}$.
\end{prop}
Proposition \ref{prop:foliationbysubmanifold}, \ref{prop:leavesaredefinable} and the definable Chow theorem \cite{PS09} will imply Theorem \ref{Thm:foliationbyalgsubvar}. From this point, by definable we shall mean definable in the o-minimal structure $\bR_{\mathrm{an, exp}}$ unless a different o-minimal structure is specified.

\subsection{Proof of Proposition \ref{prop:leavesaredefinable}}
By using the $C^{\infty}$-isomorphism $\cJ(\cV)\simeq \cJ(\cV_\bR)$ we may regard $\nu$ as an element in $H^0(B, \cJ(\cV_\bR))$. 

Choose a projective completion $\hat{B}$ of $B$ such that $\hat{B}-B$ is a simple normal crossing divisor. We may take a finite open cover $\{U_i, \ i\in \mathfrak{I}\}$ of $\hat{B}$ such that $U_i\cap B\simeq (\Delta^*)^{k_i}\times (\Delta)^{r-k_i}$ for each $i$. For any complex analytic leaf $\zeta$ of $B$, showing $\zeta\cap U_i$ has a definable structure compatible with $U_i$ is enough for showing Proposition \ref{prop:leavesaredefinable}. We thus reduce to showing the case when $\nu$ is a normal function defined over $(\Delta^*)^n$.

Consider the associated mixed period map and its lift to $\cH^n$, where $\cH$ is the Siegel upper half space:
\begin{equation}
\begin{tikzcd}
\cH^n \arrow[d, "\mathrm{exp}(2\pi i\cdot)"] \arrow[r, "\tilde{\Phi}"] & \cD\simeq D\times \mathcal{S}(W) \arrow[d] \\
(\Delta^*)^n \arrow[r, "\Phi"] & \Gamma \backslash \cD
\end{tikzcd}
\end{equation}
The following two results are critical:
\begin{prop}\cite[Prop. 5.2-5.3]{BBKT24}
    The image of $\tilde{\Phi}(R^n)\rightarrow \cD$ is contained in a (finite union of) fundamental set of $\Gamma$ acting on $\cD$, where $R^n\subset \cH^n$ is a subset of the form:
    \begin{equation}
        \{(z_1,...,z_n)\in \cH^n \ | \ |\mathfrak{Re}(z_j)|\leq M, \ \mathfrak{Im}(z_j)\geq N\}
    \end{equation}
    for some $M,N\in \bR_+$. In particular, the composition $\tilde{\Phi}(R^n)\rightarrow \cD\rightarrow \Gamma\backslash \cD$ is definable.
\end{prop}
\begin{prop}\cite[Prop. 6.5]{BBKT24}
    Using $\cD\simeq D\times \mathcal{S}(W)$, if $\mathfrak{S}\subset D$ is a fundamental set for the $\Gamma_0$-action on $D$ and $\Sigma\subset \mathcal{S}(W)$ is a bounded semi-algebraic subset, then $\mathfrak{S}\times \Sigma\rightarrow \Gamma\backslash \cD$ is definable in the o-minimal structure $\bR_{\mathrm{alg}}$.
\end{prop}
We are ready to finish our proof. Note that in the normal function case $\mathcal{S}(W)\simeq V_\bR$. For any $w\in V_\bR$ consider the subset $\cD_w:=D\times \{w\}\subset \cD$. $\cD_w$ is an $\bR_{\mathrm{alg}}$-definable subset in $\cD$, therefore $\tilde{\Phi}(R^n)\cap \cD_w$ is a definable subset in $\cD$ being contained in some fundamental set $\mathfrak{S}$ of $\Gamma$.

To sum up, choose $M,N$ properly such that $R^n\rightarrow (\Delta^*)^n$ is surjective, we have a diagram lies in the category of definable analytic spaces:
\begin{equation}
\begin{tikzcd}
R^n \arrow[d, "\mathrm{exp}(2\pi i\cdot)"] \arrow[r, "\tilde{\Phi}"] & \tilde{\Phi}(R^n) \arrow[d] & \tilde{\Phi}(R^n)\cap (D\times \{w\}) \arrow[l] \arrow[ld]\\
(\Delta^*)^n \arrow[r, "\Phi"] &  \Phi((\Delta^*)^n)
\end{tikzcd}
\end{equation}
By definition, a leaf $\zeta\subset (\Delta^*)^n$ given by Proposition \ref{prop:foliationbysubmanifold} is exactly the projection of $\tilde{\Phi}^{-1}(\tilde{\Phi}(R^n)\cap (D\times \{w\}))$ on $(\Delta^*)^n$ for some $w\in V_\bR$ (in which case $w$ is the constant lift of $\nu|_\zeta$), therefore must be definable.


\

\noindent\textbf{Conflict of interest statement.} The author states that there is no conflict of interest.

 \ 

 \noindent\textbf{Data availability statement.} The author states that this is not applicable, as there are no associated data.

\printbibliography

@article{And92,
author = {Yves André},
title = {Mumford-Tate groups of mixed Hodge structures and the theorem of the fixed part},
journal = {Compositio Mathematica},
volume = {82},
year = {1992},
number = {1},
pages = {1-24}
}

@article{BBT22,
author = {Benjamin Bakker and Yohan Brunebarbe and Jacob Tsimerman},
title = {o-minimal GAGA and a conjecture of Griffiths},
journal = {Inventiones mathematicae},
volume = {232},
year = {2022},
pages = {163-228}
}

@article{BBKT24,
author = {Benjamin Bakker and Yohan Brunebarbe and Bruno Klingler and Jacob Tsimerman},
title = {Definability of mixed period maps},
journal = {J. Eur. Math. Soc.},
volume = {26},
number = {6},
year = {2024},
pages = {2191-2209}
}

@article{BKT20,
author = { B. Bakker and B. Klingler and J. Tsimerman},
title = {Tame topology of arithmetic quotients and algebraicity of Hodge loci},
journal = {J. Amer. Math. Soc.},
volume = {33},
year = {2020},
pages = {917-939}
}

@incollection{BZ90,
  author      = "Jean-Luc Brylinski and Steven Zucke",
  title       = "An Overview of Recent Advances in Hodge Theory",
  
  booktitle   = "Complex Manifolds",
  publisher   = "Springer Berlin",
  address     = "Heidelberg",
  year        = "1998",
  pages       = "39-142",
}

@article{Cer83,
author = {G. Ceresa},
title = {C is not Algebraically Equivalent to C- in its Jacobian},
journal = {Annals of Mathematics},
volume = {117},
year = {1983},
number = {02},
pages = {285--291}
}

@article{CP95,
author = {A. Collino and G. Pirola},
title = {The Griffiths infinitesimal invariant for a curve in its Jacobian},
journal = {Duke Math. J.},
volume = {78},
year = {1995},
pages = {59-88}
}

@book{CSP17,
author = {J. Carlson and S. M\"uller-Stach and C. Peters},
title = {Period Mappings and Period Domains, 2nd Edition},
series = {Cambridge Studies in Advanced Mathematics},
volume = {168},
publisher = {Cambridge University Press},
year = {2017}
}

@article{Gre89,
author = {Mark Green},
title = {Griffiths' infinitesimal invariant and the Abel-Jacobi map},
journal = {J. Differential Geom.},
volume = {29},
number = {3},
year = {1989},
pages = {545-555},
}

@article{Gri83,
   author = {Phillip Griffiths},
   title={Infinitesimal variations of hodge structure (III)
: determinantal varieties and the infinitesimal
invariant of normal functions},
   journal={Compositio Mathematica},
   volume = {50},
   number = {2-3},
year = {1983},
pages = {267-324}
}

@misc{GZ24,
      title={Heights and periods of algebraic cycles in families}, 
      author={Ziyang Gao and Shou-Wu Zhang},
      year={2024},
      eprint={2407.01304},
      archivePrefix={arXiv},
      primaryClass={math.AG},
      url={https://arxiv.org/abs/2407.01304}, 
}

@article{Ha13,
author = {Richard Hain},
title = {Normal Functions and the Geometry of Moduli Spaces of Curves},
journal = {Handbook of Moduli, edited by Gavril Farkas, Ian Morrison},
volume = {1},
year = {2013},
pages = {527-578}
}

@misc{Ha24,
title={The Rank of the Normal Functions of the Ceresa and Gross--Schoen Cycles}, 
      author={Richard Hain},
      year={2024},
      eprint={2408.07809},
      archivePrefix={arXiv},
      primaryClass={math.AG},
      url={https://arxiv.org/abs/2408.07809}, 
}

@misc{KT24,
      title={On the torsion locus of the Ceresa normal function}, 
      author={Matt Kerr and Salim Tayou},
      year={2024},
      eprint={2406.19366},
      archivePrefix={arXiv},
      primaryClass={math.AG},
      url={https://arxiv.org/abs/2406.19366}, 
}

@article{Lat23,
author = {Robert Laterveer},
title = {On the tautological ring of Humbert curves},
journal = {Manuscripta Math.},
volume = {172},
year = {2023},
pages = {1093-1107}
}

@article{Loo24,
author = {Eduard Looijenga},
title = {A ball quotient parametrizing trigonal genus 4 curves},
journal = {	Nagoya Math. J.},
volume = {254},
year = {2024},
pages = {366-378}
}

@misc{LS24,
      title={Vanishing criteria for Ceresa cycles}, 
      author={Jef Laga and Ari Shnidman},
      year={2024},
      eprint={2406.03891},
      archivePrefix={arXiv},
      primaryClass={math.AG},
      url={https://arxiv.org/abs/2406.03891}, 
}

@article{McM13,
  title={Braid groups and Hodge theory},
  author={McMullen, Curtis T},
  journal={Mathematische Annalen},
  volume={355},
  number={3},
  pages={893--946},
  year={2013},
  publisher={Springer}
}

@article{No93,
author = {Madhav V. Nori},
title = {Algebraic cycles and Hodge theoretic connectivity},
journal = {Inventiones mathematicae},
volume = {111},
number = {2},
year = {1993},
pages = {349-374}
}

@article{QZ24,
author = {Congling Qiu and Wei Zhang},
title = {Vanishing results in Chow groups for the modified diagonal cycles},
journal = {Tunisian J. Math.},
volume = {6},
year = {2024},
}

@article{PS09,
author = {Ya'acov Peterzil and Sergei Starchenko},
title = {Complex analytic geometry and analytic-geometric categories},
journal = {Journal für die reine und angewandte Mathematik},
volume = {626},
year = {2009},
}

@article{PZ03,
author = {Gian Pietro Pirola and Francesco Zucconi},
title = {Variations of the Albanese morphisms},
journal = {J. Algebraic Geom.},
volume = {12},
year = {2003},
page = {535-572}
}

@article{Sai96,
author = {Morihiko Saito},
title = {Admissible normal functions},
journal = {J. Algebraic Geom.},
volume = {5},
year = {1996},
pages = {235-276}
}

@article{Sch73,
author = {Wilfried Schmid},
title = {Variation of Hodge Structure: The Singularities of the Period Mapping},
journal = {Inventiones mathematicae},
volume = {22},
year = {1973},
pages = {211-320}
}

@article{Xu18,
author = {Jinxing Xu},
title = {Zariski Density of Monodromy Groups via a Picard–Lefschetz Type Formula},
journal = {International Mathematics Research Notices},
volume = {2018},
number = {11},
year = {2018},
pages = {3556-3586}
}
\end{document}